\documentclass{llncs}
\usepackage[english]{babel}
\usepackage{amssymb,amsmath,hyperref}
\usepackage{bbold}
\usepackage{amsrefs}
\usepackage{stackrel}
\usepackage{mathrsfs}
\usepackage{cleveref, enumerate}

\usepackage{transfer}

\newtheorem{thm}{Theorem}

\newtheorem{cor}[thm]{Corollary}
\newtheorem{defi}[thm]{Definition}
\newtheorem{rem}[thm]{Remark}
\newtheorem{nota}[thm]{Notation}
\newtheorem{exa}[thm]{Example}

\newtheorem{ack}[thm]{Acknowledgement}

\newcommand\be{\begin{equation}}
\newcommand\ee{\end{equation}}

\def\bdefi{\begin{defi}\rm}
\def\edefi{\end{defi}}
\def\bnota{\begin{nota}\rm}
\def\enota{\end{nota}}
\def\brem{\begin{rem}\rm}
\def\erem{\end{rem}}

\def\H{\textup{\textsf{H}}}
\def\RCA{\textup{\textsf{RCA}}}

\def\RCAo{\textup{\textsf{RCA}}_{0}^{\omega}}
\def\RCAO{\textup{\textsf{RCA}}_{0}^{\Lambda}}

\def\ef{\textup{\textsf{ef}}}
\def\ns{\textup{\textsf{ns}}}

\def\WKL{\textup{\textsf{WKL}}}

\def\TOT{\textup{\textsf{TOT}}}
\def\ECF{\textup{\textsf{ECF}}}

\def\T{\mathcal{T}}

\def\bye{\end{document}}

\def\P{\textup{\textsf{P}}}

\def\R{{\mathbb  R}}

\def\MCT{\textup{\textsf{MCT}}}

\def\R{{\mathbb{R}}}
\def\({\textup{(}}
\def\){\textup{)}}

\def\st{\textup{st}}
\def\nst{\textup{nst}}
\def\asa{\leftrightarrow}

\def\di{\rightarrow}

\def\epsa{\varepsilon_{\alpha}}

\def\ACA{\textup{\textsf{ACA}}}
\def\paai{\Pi_{1}^{0}\textup{-\textsf{TRANS}}}

\def\QFAC{\textup{\textsf{QF-AC}}}

\def\MU{\textup{\textsf{MU}}}

\def\HAC{\textup{\textsf{HAC}}}

\def\INT{\textup{\textsf{int}}}

\numberwithin{equation}{section}
\numberwithin{thm}{section}

\begin{document}
\title{From Nonstandard Analysis to various flavours of Computability Theory}

\author{Sam Sanders\thanks{This research was supported by FWO Flanders, the John Templeton Foundation, the Alexander von Humboldt Foundation,  LMU Munich (via the Excellence Initiative), and the Japan Society for the Promotion of Science.  The work was done partially while the author was visiting the Institute for Mathematical Sciences, National University of Singapore in 2016. The visit was supported by the Institute.}}
\institute{Munich Center for Mathematical Philosophy, LMU Munich, Germany \& Department of Mathematics, Ghent University, Belgium
\email{sasander@me.com}}

\maketitle

\begin{abstract}
As suggested by the title, it has recently become clear that theorems of \emph{Nonstandard Analysis} (NSA) give rise to theorems in computability theory (no longer involving NSA).   
Now, the aforementioned discipline divides into \emph{classical} and \emph{higher-order} computability theory, where the former (resp.\ the latter) sub-discipline deals with objects of type zero and one (resp.\ of all types).  
The aforementioned results regarding NSA deal exclusively with the \emph{higher-order} case; we show in this paper that theorems of NSA also give rise to theorems in \emph{classical} computability theory by considering so-called \emph{textbook proofs}.  
\end{abstract}
This paper will appear in the proceedings of TAMC2017 (\cite{TAMC})

\section{Introduction}\label{intro}
Computability theory naturally\footnote{The distinction `classical versus higher-order' is not binary as e.g.\ continuous functions on the reals may be represented by type one objects (See e.g.\ \cite{simpson2}*{II.6.1}).} includes two sub-disciplines: \emph{classical} and \emph{higher-order} computability theory.  
The former deals with the computability of objects of types zero and one (natural numbers and sets thereof) and the latter deals with the computability of \emph{higher-order objects}, i.e.\ including objects of type higher than zero and one.
Friedman's closely related foundational program \emph{Reverse Mathematics} (RM for short; see \cite{simpson2,simpson1} for an overview) makes use of \emph{second-order arithmetic} which is also limited to type zero and one objects; Kohlenbach has introduced \emph{higher-order} RM in which all finite types are available (\cite{kohlenbach2}).  

\medskip

As developed in \cite{sambon, samGH, samzoo, samzooII}, one can extract \emph{higher-order} computability results from theorems in Nonstandard Analysis.  These results (\cite{samzoo,samzooII, sambon}) involve the `Big Five' of RM and also the associated `RM zoo' from \cite{damirzoo}, but all results are part of \emph{higher-order} RM.  
The following question thus naturally emerges:  
\begin{quote}
\textsf{(Q)} \emph{Is it possible to obtain \emph{classical} computability theoretic results, including second-order Reverse Mathematics, from} NSA?
\end{quote}
We will provide a positive answer to the question \textsf{(Q)} in this paper by studying an example based on the \emph{monotone convergence theorem} in Section \ref{YY}, after introducing Nonstandard Analysis and an essential fragment in Section \ref{P}.  The notion \emph{textbook proof} plays an important role.      
We also argue that our example generalises to many results in Nonstandard Analysis, as will be explored in \cite{moooore}.  

\medskip

Finally, we stress that our final results in (classical) computability theory are extracted \emph{directly} from \emph{existing theorems} of Nonstandard Analysis \textbf{without} modifications (involving computability theory or otherwise).  In particular, no modification is made to the proofs or theorems in Nonstandard Analysis.  We do consider special proofs in Nonstandard Analysis, which we christen \emph{textbook proofs} due to their format.  One could obtain the same results by mixing Nonstandard Analysis and computability theory, but one of the conceptual goals of our paper is to show that classical computability theory is \emph{already implicit} in Nonstandard Analysis \emph{pur sang}.            

\section{Internal set theory and its fragments}\label{P}
We discuss Nelson's axiomatic Nonstandard Analysis from \cite{wownelly}, and the fragment $\P$ from \cite{brie}.  
The fragment $\P$ is essential to our enterprise due to Corollary~\ref{consresultcor}.  
\subsection{Internal set theory 101}\label{IIST}
In Nelson's \emph{syntactic} (or \emph{axiomatic}) approach to Nonstandard Analysis (\cite{wownelly}), as opposed to Robinson's semantic one (\cite{robinson1}), a new predicate `st($x$)', read as `$x$ is standard' is added to the language of \textsf{ZFC}, the usual foundation of mathematics.  
The notations $(\forall^{\st}x)$ and $(\exists^{\st}y)$ are short for $(\forall x)(\st(x)\di \dots)$ and $(\exists y)(\st(y)\wedge \dots)$.  A formula is called \emph{internal} if it does not involve `st', and \emph{external} otherwise.   
The three external axioms \emph{Idealisation}, \emph{Standard Part}, and \emph{Transfer} govern the new predicate `st';  They are respectively defined\footnote{The superscript `fin' in \textsf{(I)} means that $x$ is finite, i.e.\ its number of elements are bounded by a natural number.} as:  
\begin{enumerate}
\item[\textsf{(I)}] $(\forall^{\st~\textup{fin}}x)(\exists y)(\forall z\in x)\varphi(z,y)\di (\exists y)(\forall^{\st}x)\varphi(x,y)$, for any internal $\varphi$.  
\item[\textsf{(S)}] $(\forall^{\st} x)(\exists^{\st}y)(\forall^{\st}z)\big((z\in x\wedge \varphi(z))\asa z\in y\big)$, for any $\varphi$.
\item[\textsf{(T)}] $(\forall^{\st}t)\big[(\forall^{\st}x)\varphi(x, t)\di (\forall x)\varphi(x, t)\big]$, where $\varphi(x,t)$ is internal, and only has free variables $t, x$.  
\end{enumerate}
The system \textsf{IST} is (the internal system) \textsf{ZFC} extended with the aforementioned external axioms;  
The former is a conservative extension of \textsf{ZFC} for the internal language, as proved in \cite{wownelly}.    

\medskip

In \cite{brie}, the authors study fragments of \textsf{IST} based on Peano and Heyting arithmetic.  
In particular, they consider the systems $\H$ and $\P$, introduced in the next section, which are conservative extensions of the (internal) logical systems \textsf{E-HA}$^{\omega}$ and $\textsf{E-PA}^{\omega}$, respectively \emph{Heyting and Peano arithmetic in all finite types and the axiom of extensionality}.       
We refer to \cite{kohlenbach3}*{\S3.3} for the exact definitions of the (mainstream in mathematical logic) systems \textsf{E-HA}$^{\omega}$ and $\textsf{E-PA}^{\omega}$.  
Furthermore, \textsf{E-PA}$^{\omega*}$ and $\textsf{E-HA}^{\omega*}$ are the definitional extensions of \textsf{E-PA}$^{\omega}$ and $\textsf{E-HA}^{\omega}$ with types for finite sequences, as in \cite{brie}*{\S2}.  For the former, we require some notation.  
\begin{nota}[Finite sequences]\label{skim}\rm
The systems $\textsf{E-PA}^{\omega*}$ and $\textsf{E-HA}^{\omega*}$ have a dedicated type for `finite sequences of objects of type $\rho$', namely $\rho^{*}$.  Since the usual coding of pairs of numbers goes through in both, we shall not always distinguish between $0$ and $0^{*}$.  
Similarly, we do not always distinguish between `$s^{\rho}$' and `$\langle s^{\rho}\rangle$', where the former is `the object $s$ of type $\rho$', and the latter is `the sequence of type $\rho^{*}$ with only element $s^{\rho}$'.  The empty sequence for the type $\rho^{*}$ is denoted by `$\langle \rangle_{\rho}$', usually with the typing omitted.  

\medskip

Furthermore, we denote by `$|s|=n$' the length of the finite sequence $s^{\rho^{*}}=\langle s_{0}^{\rho},s_{1}^{\rho},\dots,s_{n-1}^{\rho}\rangle$, where $|\langle\rangle|=0$, i.e.\ the empty sequence has length zero.
For sequences $s^{\rho^{*}}, t^{\rho^{*}}$, we denote by `$s*t$' the concatenation of $s$ and $t$, i.e.\ $(s*t)(i)=s(i)$ for $i<|s|$ and $(s*t)(j)=t(|s|-j)$ for $|s|\leq j< |s|+|t|$. For a sequence $s^{\rho^{*}}$, we define $\overline{s}N:=\langle s(0), s(1), \dots,  s(N)\rangle $ for $N^{0}<|s|$.  
For $\alpha^{0\di \rho}$, we also write $\overline{\alpha}N=\langle \alpha(0), \alpha(1),\dots, \alpha(N)\rangle$ for \emph{any} $N^{0}$.  By way of shorthand, $q^{\rho}\in Q^{\rho^{*}}$ abbreviates $(\exists i<|Q|)(Q(i)=_{\rho}q)$.  Finally, we shall use $\underline{x}, \underline{y},\underline{t}, \dots$ as short for tuples $x_{0}^{\sigma_{0}}, \dots x_{k}^{\sigma_{k}}$ of possibly different type $\sigma_{i}$.          
\end{nota}    
%
%
\subsection{The classical system $\P$}\label{PIPI}
In this section, we introduce $\P$, a conservative extension of $\textsf{E-PA}^{\omega}$ with fragments of Nelson's $\IST$.  
We first introduce the system $\textsf{E-PA}_{\st}^{\omega*}$ using the definition from \cite{brie}*{Def.~6.1}.  Recall that \textsf{E-PA}$^{\omega*}$ is the definitional extension of \textsf{E-PA}$^{\omega}$ with types for finite sequences as in \cite{brie}*{\S2} and Notation \ref{skim}.  
The language of $\textsf{E-PA}_{\st}^{\omega*}$ is that of $\textsf{E-PA}^{\omega*}$ extended with new symbols $\st_{\sigma}$ for any finite type $\sigma$ in $\textsf{E-PA}^{\omega*}$.  
\begin{nota}\label{notawin}\rm
We write $(\forall^{\st}x^{\tau})\Phi(x^{\tau})$ and $(\exists^{\st}x^{\sigma})\Psi(x^{\sigma})$ for 
$(\forall x^{\tau})\big[\st(x^{\tau})\di \Phi(x^{\tau})\big]$ and $(\exists x^{\sigma})\big[\st(x^{\sigma})\wedge \Psi(x^{\sigma})\big]$.     
A formula $A$ is `internal' if it does not involve `$\st$', and external otherwise.   The formula $A^{\st}$ is defined from internal $A$ by appending `st' to all quantifiers (except bounded number quantifiers).    
\end{nota}
The set $\T^{*}$ is defined as the collection of all the terms in the language of $\textsf{E-PA}^{\omega*}$.  
\bdefi\label{debs}
The system $ \textsf{E-PA}^{\omega*}_{\st} $ is defined as $ \textsf{E-PA}^{\omega{*}} + \T^{*}_{\st} + \textsf{IA}^{\st}$, where $\T^{*}_{\st}$
consists of the following axiom schemas.
\begin{enumerate}
\item The schema\footnote{The language of $\textsf{E-PA}_{\st}^{\omega*}$ contains a symbol $\st_{\sigma}$ for each finite type $\sigma$, but the subscript is essentially always omitted.  Hence $\T^{*}_{\st}$ is an \emph{axiom schema} and not an axiom.\label{omit}} $\st(x)\wedge x=y\di\st(y)$,
\item The schema providing for each closed\footnote{A term is called \emph{closed} in \cite{brie} (and in this paper) if all variables are bound via lambda abstraction.  Thus, if $\underline{x}, \underline{y}$ are the only variables occurring in the term $t$, the term $(\lambda \underline{x})(\lambda\underline{y})t(\underline{x}, \underline{y})$ is closed while $(\lambda \underline{x})t(\underline{x}, \underline{y})$ is not.  The second axiom in Definition \ref{debs} thus expresses that $\st_{\tau}\big((\lambda \underline{x})(\lambda\underline{y})t(\underline{x}, \underline{y})\big)$ if $(\lambda \underline{x})(\lambda\underline{y})t(\underline{x}, \underline{y})$ is of type $\tau$.\label{kootsie}} term $t\in \T^{*}$ the axiom $\st(t)$.
\item The schema $\st(f)\wedge \st(x)\di \st(f(x))$.
\end{enumerate}
The external induction axiom \textsf{IA}$^{\st}$ states that for any (possibly external) $\Phi$:
\be\tag{\textsf{IA}$^{\st}$}
\Phi(0)\wedge(\forall^{\st}n^{0})(\Phi(n) \di\Phi(n+1))\di(\forall^{\st}n^{0})\Phi(n).     
\ee
\edefi
Secondly, we introduce some essential fragments of $\IST$ studied in \cite{brie}.  
\bdefi[External axioms of $\P$]~
\begin{enumerate}
\item$\HAC_{\INT}$: For any internal formula $\varphi$, we have
\be\label{HACINT}
(\forall^{\st}x^{\rho})(\exists^{\st}y^{\tau})\varphi(x, y)\di \big(\exists^{\st}F^{\rho\di \tau^{*}}\big)(\forall^{\st}x^{\rho})(\exists y^{\tau}\in F(x))\varphi(x,y),
\ee
\item $\textsf{I}$: For any internal formula $\varphi$, we have
\[
(\forall^{\st} x^{\sigma^{*}})(\exists y^{\tau} )(\forall z^{\sigma}\in x)\varphi(z,y)\di (\exists y^{\tau})(\forall^{\st} x^{\sigma})\varphi(x,y), 
\]
\item The system $\P$ is $\textsf{E-PA}_{\st}^{\omega*}+\textsf{I}+\HAC_{\INT}$.
\end{enumerate}
\end{defi}
Note that \textsf{I} and $\HAC_{\INT}$ are fragments of Nelson's axioms \emph{Idealisation} and \emph{Standard part}.  
By definition, $F$ in \eqref{HACINT} only provides a \emph{finite sequence} of witnesses to $(\exists^{\st}y)$, explaining its name \emph{Herbrandized Axiom of Choice}.   

\medskip

The system $\P$ is connected to $\textsf{E-PA}^{\omega}$ by the following theorem.    
Here, the superscript `$S_{\st}$' is the syntactic translation defined in \cite{brie}*{Def.\ 7.1}.  
\begin{thm}\label{consresult}
Let $\Phi(\tup a)$ be a formula in the language of \textup{\textsf{E-PA}}$^{\omega*}_{\st}$ and suppose $\Phi(\tup a)^\Sh\equiv\forallst \tup x \, \existsst \tup y \, \varphi(\tup x, \tup y, \tup a)$. If $\Delta_{\intern}$ is a collection of internal formulas and
\be\label{antecedn}
\P + \Delta_{\intern} \vdash \Phi(\tup a), 
\ee
then one can extract from the proof a sequence of closed\footnote{Recall the definition of closed terms from \cite{brie} as sketched in Footnote \ref{kootsie}.\label{kootsie2}} terms $t$ in $\mathcal{T}^{*}$ such that
\be\label{consequalty}
\textup{\textsf{E-PA}}^{\omega*} + \Delta_{\intern} \vdash\  \forall \tup x \, \exists \tup y\in \tup t(\tup x)\ \varphi(\tup x,\tup y, \tup a).
\ee
\end{thm}
\begin{proof}
Immediate by \cite{brie}*{Theorem 7.7}.  
\end{proof}
The proofs of the soundness theorems in \cite{brie}*{\S5-7} provide an algorithm $\mathcal{A}$ to obtain the term $t$ from the theorem.  In particular, these terms 
can be `read off' from the nonstandard proofs.    

\medskip

In light of \cite{sambon}, the following corollary (which is not present in \cite{brie}) is essential to our results.  Indeed, the following corollary expresses that we may obtain effective results as in \eqref{effewachten} from any theorem of Nonstandard Analysis which has the same form as in \eqref{bog}.  It was shown in \cite{sambon, samzoo, samzooII} that the scope of this corollary includes the Big Five systems of RM and the RM `zoo' (\cite{damirzoo}).  
\begin{cor}\label{consresultcor}
If $\Delta_{\intern}$ is a collection of internal formulas and $\psi$ is internal, and
\be\label{bog}
\P + \Delta_{\intern} \vdash (\forall^{\st}\underline{x})(\exists^{\st}\underline{y})\psi(\underline{x},\underline{y}, \underline{a}), 
\ee
then one can extract from the proof a sequence of closed$^{\ref{kootsie2}}$ terms $t$ in $\mathcal{T}^{*}$ such that
\be\label{effewachten}
\textup{\textsf{E-PA}}^{\omega*} +\QFAC^{1,0}+ \Delta_{\intern} \vdash (\forall \underline{x})(\exists \underline{y}\in t(\underline{x}))\psi(\underline{x},\underline{y},\underline{a}).
\ee
\end{cor}
\begin{proof}
Clearly, if for internal $\psi$ and $\Phi(\underline{a})\equiv (\forall^{\st}\underline{x})(\exists^{\st}\underline{y})\psi(x, y, a)$, we have $[\Phi(\underline{a})]^{S_{\st}}\equiv \Phi(\underline{a})$, then the corollary follows immediately from the theorem.  
A tedious but straightforward verification using the clauses (i)-(v) in \cite{brie}*{Def.\ 7.1} establishes that indeed $\Phi(\underline{a})^{S_{\st}}\equiv \Phi(\underline{a})$.  
\end{proof}
For the rest of this paper, the notion `normal form' shall refer to a formula as in \eqref{bog}, i.e.\ of the form $(\forall^{\st}x)(\exists^{\st}y)\varphi(x,y)$ for $\varphi$ internal.  

\medskip

Finally, we will use the usual notations for natural, rational and real numbers and functions as introduced in \cite{kohlenbach2}*{p.\ 288-289}. (and \cite{simpson2}*{I.8.1} for the former).  
We only list the definition of real number and related notions in $\P$.
\begin{defi}[Real numbers and related notions in $\P$]\label{keepintireal}\rm~
\begin{enumerate}
\item A (standard) real number $x$ is a (standard) fast-converging Cauchy sequence $q_{(\cdot)}^{1}$, i.e.\ $(\forall n^{0}, i^{0})(|q_{n}-q_{n+i})|<_{0} \frac{1}{2^{n}})$.  
We use Kohlenbach's `hat function' from \cite{kohlenbach2}*{p.\ 289} to guarantee that every sequence $f^{1}$ is a real.  
\item We write $[x](k):=q_{k}$ for the $k$-th approximation of a real $x^{1}=(q^{1}_{(\cdot)})$.    
\item Two reals $x, y$ represented by $q_{(\cdot)}$ and $r_{(\cdot)}$ are \emph{equal}, denoted $x=_{\R}y$, if $(\forall n)(|q_{n}-r_{n}|\leq \frac{1}{2^{n}})$. Inequality $<_{\R}$ is defined similarly.         
\item We  write $x\approx y$ if $(\forall^{\st} n)(|q_{n}-r_{n}|\leq \frac{1}{2^{n}})$ and $x\gg y$ if $x>y\wedge x\not\approx y$.  
\item A function $F:\R\di \R$ mapping reals to reals is represented by $\Phi^{1\di 1}$ mapping equal reals to equal reals as in $(\forall x, y)(x=_{\R}y\di \Phi(x)=_{\R}\Phi(y))$.
\item We write `$N\in \Omega$' as a \emph{symbolic} abbreviation for $\neg\st(N^{0})$.  
\end{enumerate}
\end{defi}
\section{Main results}\label{YY}
In this section, we provide an answer to the question \textsf{(Q)} from Section~\ref{intro} by studying the \emph{monotone convergence theorem}.  
We first obtain the associated result in higher-order computability theory from NSA in Section \ref{EXA}.  
We then establish in Section \ref{post} that the same proof in NSA also gives rise to \emph{classical} computability theory.  

\subsection{An example of the computational content of NSA}\label{EXA}
In this section, we provide an example of the \emph{higher-order} computational content of NSA, involving the \emph{monotone convergence theorem}, $\MCT$ for short, which is the statement \emph{every monotone sequence in the unit interval converges}.  
In particular, we consider the equivalence between a nonstandard version of $\MCT$ and a fragment of Nelson's axiom \emph{Transfer} from Section \ref{P}.
From this nonstandard equivalence, an explicit RM equivalence involving higher-order versions of $\MCT$ and arithmetical comprehension is extracted as in \eqref{frood}.  

\medskip

Firstly, nonstandard $\MCT$ (involving \emph{nonstandard} convergence) is:
\be\label{MCTSTAR}\tag{\MCT$_{\textsf{ns}}$}
(\forall^{\st} c_{(\cdot)}^{0\di 1})\big[(\forall n^{0})(c_{n}\leq c_{n+1}\leq 1)\di (\forall N,M\in \Omega)[c_{M}\approx c_{N}]    \big].
\ee
while the effective (or `uniform') version of $\MCT$, abbreviated \MCT$_{\textsf{ef}}(t)$, is:
\[\textstyle
(\forall c_{(\cdot)}^{0\di 1},k^{0})\big[(\forall n^{0})(c_{n}\leq c_{n+1}\leq 1)\di (\forall N,M\geq t(c_{(\cdot)})(k))[|c_{M}- c_{N}|\leq \frac{1}{k} ]   \big].
\]
We require two equivalent (\cite{kohlenbach2}*{Prop.\ 3.9}) versions of arithmetical comprehension, respectively the \emph{Turing jump functional} and \emph{Feferman's mu-operator}, as follows  
\be\label{mukio}\tag{$\exists^{2}$}
(\exists \varphi^{2})\big[(\forall f^{1})( (\exists n^{0})f(n)=0 \asa \varphi(f)=0    \big],
\ee
\be\label{mu}\tag{$\mu^{2}$}
(\exists \mu^{2})\big[(\forall f^{1})( (\exists n^{0})f(n)=0 \di f(\mu(f))=0)    \big],
\ee
and also the restriction of Nelson's axiom \emph{Transfer} as follows:
\be\tag{$\paai$}
(\forall^{\st}f^{1})\big[(\forall^{\st}n^{0})f(n)\ne0\di (\forall m)f(m)\ne0\big].
\ee
Denote by $\textsf{MU}(\mu)$ the formula in square brackets in \eqref{mu}.  
We have the following nonstandard equivalence.  
\begin{thm}\label{partje}
The system $\P$ proves that $\paai\asa \MCT_{\ns}$.  
\end{thm}
\begin{proof}
For the forward implication, assume $\paai$ and suppose $\MCT_{\ns}$ is false, i.e.\ there is a standard monotone sequence $c_{(\cdot)}$ such that $c_{N_{0}}\not\approx c_{M_{0}}$ for fixed nonstandard $N_{0}, M_{0}$.
The latter is by definition $|c_{N_{0}}- c_{M_{0}}|\geq \frac{1}{k_{0}}$, where $k_{0}^{0}$ is a fixed standard number. 
Since $N_{0}, M_{0}$ are nonstandard in the latter, we have $(\forall^{\st}n)(\exists N, M\geq n)
(|c_{N}- c_{M}|\geq \frac{1}{k_{0}})$.  Fix standard $n^{0}$ in the latter and note that the resulting $\Sigma_{1}^{0}$-formula only involves \emph{standard} parameters.
Hence, applying the contraposition of $\paai$, we obtain $(\forall^{\st}n)(\exists^{\st} N, M\geq n)(|c_{N}- c_{M}|\geq \frac{1}{k_{0}})$.  
Applying\footnote{To `apply this formula $k_{0}+1$ times', apply $\HAC_{\INT}$ to $(\forall^{\st}n)(\exists^{\st} N, M\geq n)(|c_{N}- c_{M}|\geq \frac{1}{k_{0}})$ to obtain standard $F^{0\di 0^{*}}$ and define $G(n)$ as the maximum of $F(n)(i)$ for $i<|F(n)|$.  Then   
$(\forall^{\st}n)(\exists  N, M\geq n)(N, M\leq G(n)\wedge |c_{N}- c_{M}|\geq \frac{1}{k_{0}})$ and iterate the functional $G$ at least $k_{0}+1$ times to obtain the desired contradiction.} the previous formula $k_{0}+1$ times would make $c_{(\cdot)}$ escape the unit interval, a contradiction; $\MCT_{\ns}$ follows and the forward implication holds.    

\medskip
\noindent
For the reverse implication, assume $\MCT_{\ns}$, fix standard $f^{1}$ such that $(\forall^{\st}n^{0})(f(n)\ne 0)$ and define the sequence $c_{(\cdot)}^{1}$ as follows: $c_{k}$ is $0$ if $(\forall n\leq k)(f(n)\ne 0)$ and $\sum_{i=1}^{k}\frac{1}{2^{i}}$ otherwise.  Note that $c_{(\cdot)}$ is standard (as $f^{1}$ is) and weakly increasing.  Hence, $c_{N}\approx c_{M}$ for nonstandard $N,M$ by $\MCT_{\ns}$.  
Now suppose $m_{0}$ is such that $f(m_{0})=0$ and also the least such number.  By the definition of $c_{(\cdot)}$, we have $0=c_{m_{0}-1}\not\approx c_{m_{0}}=\sum_{i=1}^{m_{0}}\frac{1}{2^{i}}\approx 1$.  This contradiction implies that $(\forall n^{0})(f(n)\ne 0)$, and $\paai $ thus follows.  \qed
\end{proof}
We refer to the previous proof as the `textbook proof' of $\MCT_{\ns}\asa \paai$.  The reverse implication is indeed very similar to the proof of $\MCT\di \ACA_{0}$ in Simpson's textbook on RM, as found in \cite{simpson2}*{I.8.4}.  This `textbook proof' is special in a specific sense, as will become clear in the next section.     
Nonetheless, \textbf{any} nonstandard proof will yield higher-order computability results as in \eqref{frood}.  
\begin{thm}\label{sef}
From \textbf{any} proof of $\MCT_{\ns}\asa \paai$ in $\P $, two terms $s, u$ can be extracted such that $\textup{\textsf{E-PA}}^{\omega*}$ proves:
\be\label{frood}
(\forall \mu^{2})\big[\textsf{\MU}(\mu)\di \MCT_{\ef}(s(\mu)) \big] \wedge (\forall t^{1\di 1})\big[ \MCT_{\ef}(t)\di  \MU(u(t))  \big].
\ee
\end{thm}
\begin{proof}
We prove the second conjunct and leave the first one to the reader.  Corollary~\ref{consresultcor} only applies to normal forms and we now bring $\MCT_{\ns}\di \paai$ into a suitable normal form to apply this corollary and obtain the second conjunct of \eqref{frood}.  Clearly, $\paai$ implies the following normal form:
\be\label{tilda}
(\forall^{\st}f^{1})(\exists^{\st}n^{0})\big[ (\exists m)f(m)=0)\di f(n)=0\big].
\ee
The nonstandard convergence of $c_{(\cdot)}$, namely $(\forall N,M\in \Omega)[c_{M}\approx c_{N}]$, implies 
\[\textstyle
(\forall N,M)[ (\forall^{\st}n^{0})(M,N\geq n)\di   (\forall^{\st}k)|c_{M}- c_{N}|<\frac{1}{k}],
\]
in which we pull the standard quantifiers to the front as follows:
\[\textstyle
(\forall^{\st}k^{0})\underline{(\forall N,M)(\exists^{\st}n^{0})[ M,N\geq n\di   |c_{M}- c_{N}|<\frac{1}{k}]},
\]  
The contraposition of \emph{idealisation} \textsf{I} applies to the underlined.  We obtain:
\[\textstyle
(\forall^{\st}k^{0})(\exists^{\st}z^{0^{*}}){(\forall N,M)(\exists n^{0}\in z)[ M,N\geq n\di   |c_{M}- c_{N}|<\frac{1}{k}]},
\]  
and define $K^{0}$ as the maximum of $z(i)$ for $i<|z|$.  We finally obtain:
\be\label{norma}\textstyle
(\forall^{\st}k^{0})(\exists^{\st}K^{0}){(\forall N,M)[ M,N\geq K\di   |c_{M}- c_{N}|<\frac{1}{k}]}.
\ee
and \eqref{norma} is a normal form for nonstandard convergence.  Hence, $\MCT_{\ns}$ implies:
\[\textstyle
(\forall^{\st} c_{(\cdot)}^{0\di 1},k^{0})(\exists^{\st}K^{0})\big[(\forall n^{0})(c_{n}\leq c_{n+1}\leq 1)\di (\forall N,M\geq K)[|c_{M}- c_{N}|\leq \frac{1}{k} ]   \big], 
\]
and let the formula in square brackets be $D(c_{(\cdot)}{, k, K})$, while the formula in square brackets in \eqref{tilda} is $E(f,n)$.  
Then $\MCT_{\ns}\di \paai$ implies that
\be\label{flug}
(\forall^{\st} c_{(\cdot)}^{0\di 1},k^{0})(\exists^{\st}K^{0})D(c_{(\cdot)}, k, K)\di (\forall^{\st}f^{1})(\exists^{\st}n^{0})E(f, n).  
\ee
By the basic axioms in Definition \ref{debs}, any \emph{standard} functional $\Psi$ produces standard output on standard input, which yields
\be\label{kolli}
(\forall^{\st}\Psi) \big[(\forall^{\st} c_{(\cdot)}^{0\di 1},k^{0})D(c_{(\cdot)}, k, \Psi(k, c_{(\cdot)}))\di (\forall^{\st}f^{1})(\exists^{\st}n^{0})E(f, n)\big].  
\ee
We may drop the remaining `st' in the antecedent of \eqref{kolli} to obtain:
\[
(\forall^{\st}\Psi) \big[(\forall c_{(\cdot)}^{0\di 1},k^{0})D(c_{(\cdot)}, k, \Psi(k, c_{(\cdot)}))\di (\forall^{\st}f^{1})(\exists^{\st}n^{0})E(f, n)\big], 
\]
and bringing all standard quantifiers to the front, we obtain a normal form:
\be\label{hoori}
(\forall^{\st}\Psi , f^{1})(\exists^{\st}n^{0})\big[(\forall c_{(\cdot)}^{0\di 1},k^{0})D(c_{(\cdot)}, k, \Psi(k, c_{(\cdot)}))\di E(f, n)\big].  
\ee
Applying Corollary \ref{consresultcor} to `$\P\vdash \eqref{hoori}$', we obtain a term $t$ such that
\be\label{hoori12}
(\forall \Psi , f^{1})(\exists n^{0}\in t(\Psi, f))\big[(\forall c_{(\cdot)}^{0\di 1},k^{0})D(c_{(\cdot)}, k, \Psi(k, c_{(\cdot)}))\di E(f, n)\big].  
\ee
Define $s(f, \Psi)$ as the maximum of $t(\Psi, f)(i)$ for $i<|t(\Psi, f)|$.   Then \eqref{hoori} implies
\be\label{kurk}
(\forall \Psi )\big[(\forall c_{(\cdot)}^{0\di 1},k^{0})D(c_{(\cdot)}, k, \Psi(k, c_{(\cdot)}))\di (\forall f^{1})(\exists n\leq s(f, \Psi))E(f, n)\big], 
\ee
and we recognise the antecedent as the effective version of $\MCT$; the consequent is (essentially) $\MU(s(f, \Psi))$.  
Hence, the second conjunct of \eqref{frood} follows. 
\qed
\end{proof}
Note that the normal form \eqref{norma} of nonstandard convergence is the `epsilon-delta' definition of convergence with the `epsilon' and `delta' quantifiers enriched with `st'.  
While the previous proof may seem somewhat magical upon first reading, one readily jumps from the nonstandard implication $\MCT_{\ns}\di \paai$ to \eqref{kurk} with some experience.  

\medskip

In conclusion, \textbf{any} proof of $\paai\asa \MCT_{\ns}$ gives rise to the \emph{higher-order} computability result \eqref{frood}.  
We may thus conclude the latter from the proof of Theorem \ref{partje}.  In the next section, we show that the latter theorem's `textbook proof' is special in that it also gives rise to \emph{classical} computability-theoretic results.  The latter is non-trivial since both $\paai$ and $\MCT_{\ns}$ have a normal form starting with `$(\forall^{\st} h^{1})(\exists^{\st}l^{0})$' (up to coding).
As a result, to convert the implication $\MCT_{\ns}\di \paai$ into a normal form, one has to introduce a \emph{higher-order} functional like $\Psi$ to go from \eqref{flug} to \eqref{kolli}.  
Note that replacing the sequence of reals $c_{(\cdot)}^{0\di 1}$ in $\MCT_{\ns}$ by a sequence of rationals $q_{(\cdot)}^{1}$ does not lower $\Psi$ below type two. 
In a nutshell, the procedure in the previous proof (and hence most proofs in Nonstandard Analysis) \emph{always} seems to produce higher-order computability results.     

\subsection{An example of the classical-computational content of NSA}\label{post}
In the previous section, we showed that \textbf{any} proof of $\paai\asa \MCT_{\ns}$ gives rise to the \emph{higher-order} equivalence \eqref{frood}.  
In this section, we show that the particular `textbook proof' of $\paai\asa \MCT_{\ns}$ in Theorem \ref{partje} gives rise to \emph{classical} computability theoretic results as in \eqref{frood33} and \eqref{froodcor33}.  

\medskip
    
First of all, we show that the `textbook proof' of Theorem \ref{partje} is actually more uniform than the latter theorem suggests.  
To this end, let $\paai(f)$ and $\MCT_{\ns}({c_{(\cdot)}})$ be respectively $\paai$ and $\MCT_{\ns}$ from Section \ref{EXA} restricted to the function $f^{1}$ and sequence $c_{(\cdot)}$, i.e.\  the former principles are the latter with the quantifiers $(\forall f^{1})$ and $(\forall c_{(\cdot)}^{0\di 1})$ stripped off.   
\begin{thm}\label{tochwelcrux}
There are terms $s,t$ such that the system $\P$ proves 
\be\label{wagner}
(\forall^{\st}f^{1})\big[\MCT_{\ns}({t(f)})\di \paai(f)\big], 
\ee
\be\label{strauss}
(\forall^{\st} c^{0\di 1}_{(\cdot)})\big[(\forall^{\st}n^{0})\paai(s(c_{(\cdot)},n))\di \MCT_{\ns}(c_{(\cdot)})].
\ee
All proofs are implicit in the `textbook proof' of Theorem \ref{partje}.  
\end{thm}
\begin{proof}
To establish \eqref{wagner}, define the term $t^{1\di 1}$ as follows for $f^{1}, k^{0}$:   
\be\label{trufke}
t(f)(k):=
\begin{cases}
0 & (\forall i\leq k)(f(i)\ne0) \\
\sum_{i=0}^{k}\frac{1}{2^{i}} & \textup{otherwise}
\end{cases}.
\ee
The proof of Theorem \ref{partje} now yields \eqref{wagner}.  Indeed, fix a standard function $f^{1}$ such that $(\forall^{\st}k^{0})(f(k)\ne 0)\wedge (\exists n)(f(n)=0)$ and $\MCT_{\ns}(t(f))$.
By the latter, the sequence $t(f)$ nonstandard convergences, while $0=t(f)(n_{0}-1)\not\approx t(f)(n_{0})\approx 1$ for $n_{0}$ the least (necessarily nonstandard) $n$ such that $f(n)=0$.  From this contradiction, $\paai(f)$ follows, and thus also \eqref{wagner}.  

\medskip

The remaining implication \eqref{strauss} is proved in exactly the same way.  
Indeed, the intuition behind the previous part of the proof is as follows:  
In the proof of the reverse implication of Theorem \ref{partje}, to establish $\paai(f)$ for fixed standard $f^{1}$, we \textbf{only} used $\MCT_{\ns}$ for \textbf{one} particular sequence, namely $t(f)$.  
Hence, we only need $\MCT_{\ns}(t(f))$, and not `all of' $\MCT_{\ns}$, thus establishing \eqref{wagner}.  Similarly, in the proof of the forward implication of Theorem \ref{partje}, to derive $\MCT_{\ns}(c_{(\cdot)})$ for fixed $c_{(\cdot)}$, we \textbf{only} applied $\paai$ to \textbf{one} specific $\Sigma_{1}^{0}$ formula with a standard parameters $n^{0}$ and $c_{(\cdot)}$.    
\qed
\end{proof}
We are now ready to reveal the intended `deeper' meaning of the term `textbook proof':
The latter refers to a proof (which may not exist) of an implication $(\forall^{\st}f)A(f)\di (\forall^{\st} g)B(g)$ which also establishes $(\forall^{\st}g)[A(t(g))\di B(g)]$, \emph{and} in which 
the formula in square brackets is a formula in which all standard quantifiers involve variables of type zero.  By Theorem \ref{sef33}, such a `textbook proof' gives rise to results in \textbf{classical} computability theory.  

\medskip

We choose the term `textbook proof' because proofs in Nonstandard Analysis (especially in textbooks) are quite explicit in nature, i.e.\ one often establishes $(\forall^{st}g)[A(t(g))\di B(g)]$ in order to prove $(\forall^{\st}f)A(f)\di (\forall^{\st} g)B(g)$.

\medskip

Before we can apply Corollary \ref{consresultcor} to Theorem \ref{tochwelcrux}, we need some definitions, as follows.  
First, consider the following `second-order' version of $(\mu^{2})$:
\be\tag{$\MU^{A}(\nu )$}
(\forall e^{0}, n^{0})\big[(\exists m, s)(\varphi_{e,s}^{A}(n)=m)\di (\exists m, s\leq \nu(e, n))(\varphi_{e,s}^{A}(n)=m)\big].
\ee
where `$\varphi_{e,s}^{A}(m)=n$' is the usual (primitive recursive) predicate expressing that the $e$-th Turing machine with input $n$ and oracle $A$ halts after $s$ steps with output $m$; sets $A, B, C, \dots$ are denoted by binary sequences.  
One easily defines the (second-order) Turing jump of $A$ from $\nu^{1}$ as in $\MU^{A}(\nu)$ and vice versa.    

\medskip

Next, we introduce the `computability-theoretic' version of $\MCT_{\ef}(t)$.  
To this end, let $\TOT(e, A)$ be the formula `$(\forall n^{0})(\exists m^{0}, s^{0})(\varphi_{e,s}^{A}(n)=m)$', i.e.\ the formula expressing that the Turing machine with index $e$ and oracle $A$ halts for all inputs, also written `$(\forall n^{0})\varphi_{e}^{A}(n)\downarrow$'.  
Assuming the latter formula to hold for $e^{0}, A^{1}$, the function $\varphi_{e}^{A}$ is clearly well-defined, and will be used in $\P$ in the usual\footnote{For instance, written out in full `$0\leq \varphi_{e}^{A}(n)\leq \varphi_{e}^{A}(n+1)\leq 1$' from $\MCT_{\ef}^{A}(t)$ is:
\be\label{krunt}
(\forall s^{0}, q^{0}, r^{0})\big[  (\varphi_{e,s}^{A}(n)=q\wedge  \varphi_{e,s}^{A}(n+1)=r ) \di  0\leq_{0} q \leq_{0}r\leq_{0} 1 \big], 
\ee
where we also omitted the coding of rationals.} sense of computability theory.  We assume $\varphi_{e}^{A}(n)$ to code a rational number without mentioning the coding.  
We now introduce the `second-order' version of $\MCT_{\ef}(t)$:
\begin{align}\label{MCTA2}\textstyle\tag{$\MCT_{\ef}^{A}(t)$}
(\forall e^{0})\big[\TOT(e, A)\wedge (\forall n^{0})(0\leq& \varphi_{e}^{A}(n)\leq \varphi_{e}^{A}(n+1)\leq 1)\\
&\textstyle\di (\forall k^{0})(\forall N,M\geq t(e, k))[|\varphi^{A}_{e}(N)-\varphi_{e}^{A}(M)|\leq \frac{1}{k} ]   \big].\notag
\end{align}  
Here, $t$ has type $(0\times 0)\di 0$ or $0\di 1$, and we will usually treat the former as a type one object.   
Finally, let $\MCT_{\ef}^{A}(t, e)$ and $\MU^{A}(\nu, e,n)$ be the corresponding principles with the quantifiers outside the outermost square brackets removed.  

\begin{thm}\label{sef33}
From the textbook proof of $\MCT_{\ns}\di \paai$, three terms $s^{1\di 1}, u^{1}, v^{1\di 1}$ can be extracted such that $\textsf{\textup{E-PA}}^{\omega*}$ proves:
\be\label{frood33}
 (\forall A^{1}, \psi^{0\di 1})\big[ \MCT_{\ef}^{A}(\psi)\di  \MU^{A}(s( \psi, A))  \big].
\ee
\be\label{froodcor33}
 (\forall e^{0}, n^{0},A^{1},\phi^{ 1})\big[ \MCT_{\ef}^{A}(\phi, u(e,n))\di  \MU^{A}(v( \phi, A, e,n), e,n)  \big].
\ee
\end{thm}
\begin{proof}
Similar to the proof of Theorem \ref{sef}, a normal form for $\paai(f)$ is:
\be\label{noniesimpel33}
(\exists^{\st }n^{0})\big[(\exists m)(f(m)=0)\di (\exists i\leq n)(f(i)=0)  ], 
\ee
while, for $t$ as in \eqref{trufke},  a normal from for $\MCT_{\ns}(t(f))$ is:
\begin{align}\textstyle\label{noniesimpel34}
(\forall^{\st}k^{0})(\exists^{\st}K^{0})\big[(\forall n^{0} )&(0\leq t(f)(n)\leq t(f)(n+1)\leq 1)\\
&\textstyle\di (\forall N^{0}, M^{0}\geq K)(|t(f)(N)-t(f)(M)|\leq \frac{1}{k} )  \big],\notag
\end{align}
Let $C(n,f)$ (resp.\ $B(k,K, f)$) be the formula in (outermost) square brackets in \eqref{noniesimpel33} (resp.\ \eqref{noniesimpel34}).  
Then \eqref{wagner} is the formula 
\[
(\forall^{\st}f^{1})[(\forall^{\st} k)(\exists^{\st}K)B(k, K, f)\di (\exists^{\st}n^{0})C(n, f)], 
\]    
which (following the proof of Theorem \ref{sef}) readily implies the normal form:
\be\label{joki}
(\forall^{\st}f^{1}, \psi^{1})(\exists^{\st}n^{0})[(\forall k)B(k, \psi(k), f)\di C(n, f)].  
\ee 
Applying Corollary \ref{consresultcor} to `$\P_{0}\vdash \eqref{joki}$' yields a term $z^{2}$ such that 
\[
(\forall f^{1}, \psi^{1})(\exists n\in z(f, \psi))\big[(\forall k)B(k, \psi(k), f)\di C(n, f)\big]
\] 
is provable in $\textsf{E-PA}^{\omega*}$.  Define the term $s(f, \psi)$ as the maximum of all $z(f, \psi)(i)$ for $i<|z(f,\psi)|$ and note that (by the monotonicity of $C$):
\be\label{kkkk}
(\forall f^{1}, \psi^{1})\big[(\forall k)B(k, \psi(k), f)\di C(s(f, \psi), f)\big].
\ee 
Now define $f_{0}^{2}$ as follows:  $f_{0}(e, n, A, k)=0$ if $(\exists m,s\leq k)(\varphi_{e,s}^{A}(n)=m)$, and $1$ otherwise.  For this choice of function, namely taking $f^{1}=_{1}\lambda k.f_{0}$, the sentence \eqref{kkkk} implies for all $A^{1}, \psi^{1}, e^{0}, n^{0}$ that
\be\label{bahare}
(\forall k')B(k', \psi(k'), \lambda k.f_{0})\di C(s(\lambda k. f_{0}, \psi), \lambda k.f_{0}), 
\ee
where we used the familiar lambda notation with some variables of $f_{0}$ suppressed to reduce notational complexity.  
Consider the term $t$ from \eqref{trufke} and note that there are (primitive recursive) terms $x^{1}, y^{1}$ such that for all $m$ we have $t(\lambda k.f_{0}(e,n,A,k))(m)=\varphi^{A}_{x(e,n), y(e, n)}(m) $; the definition of $x^{1}, y^{1}$ is implicit in the definition of $t$ and $f_{0}$.  Hence, with these terms, the antecedent and consequent of \eqref{bahare} are as required to yield \eqref{froodcor33}.

\medskip

To prove \eqref{frood33} from \eqref{bahare}, suppose we have $(\forall k')B(k', \xi(e,n)(k'), \lambda k.f_{0})$ for all $e^{0}, n^{0}$ and some $\xi^{0\di 1} $ and $ A^{1}$, where $\xi(e, n)$ has type $1$.  By \eqref{bahare} we obtain 
\[
(\forall e^{0}, n^{0})C(s(\lambda k.f_{0}, \xi(e, n)), \lambda k.f_{0}).
\]
Putting the previous together, we obtain the sentence:
\begin{align}\label{horrenzoon}
(\forall A^{1}, \xi^{0\di 1})\big[(\forall e^{0}, n^{0}, k')B(k', &~\xi(e, n)(k'), \lambda k.f_{0})\\
&\di (\forall e^{0}, n^{0})C(s(\lambda k.f_{0},  \xi(e, n)), \lambda k.f_{0})\big].\notag  
\end{align}
Clearly, the consequent of \eqref{horrenzoon} implies that $s(\lambda k.f_{0},  \xi(e, n))$ provides the Turing jump of $A$ as in $\MU^{A}(\lambda e\lambda n.s(\lambda k.f_{0},  \xi(e, n)))$.  
On the other hand, the antecedent of \eqref{horrenzoon} expresses that the sequence $t(\lambda k.f_{0}(e, n, A, k))$ converges for all $e, n$ as witnessed by the modulus $\xi(e, n)$.  
In light of the definitions of $f_{0}$ and $t$, the sequence $t(\lambda k.f_{0})$ (considered as a type one object) is definitely computable from the oracle $A$ (in the usual sense of Turing computability).  Thus, the antecedent of \eqref{horrenzoon} also follows from $\MCT_{\ef}^{A}(\xi)$.  In other words, \eqref{horrenzoon} yields
\be
(\forall A^{1}, \xi^{0\di 1})\big[\MCT_{\ef}^{A}(\xi)\di \MU^{A}(\lambda e\lambda n.s(\lambda k.f_{0},  \xi(e, n)))\big],
\ee
which is as required for the theorem, with minor modifications to the term $s$.  
\qed
\end{proof}
Note that \eqref{froodcor33} expresses that in order to decide if the $e$-th Turing machine with oracle $A$ and input $n$ halts, it suffices to have the term $s$ and a modulus of convergence for the sequence of rationals given by $\varphi^{A}_{u(e, n)}$.  We do not claim these to be ground-breaking results in computability theory, 
but we do point out the surprising ease and elegance with which they fall out of \emph{textbook proofs} in Nonstandard Analysis.  Taking into account the claims\footnote{Bishop (See \cite{kluut}*{p.\ 513}, \cite{bishl}*{p.\ 1}, and \cite{kuddd}, which is the review of \cite{keisler3}) and Connes (See \cite{conman2}*{p.\ 6207} and \cite{conman}*{p.\ 26}) have made rather strong claims regarding the non-constructive nature of Nonstandard Analysis.  Their arguments have been investigated in remarkable detail and were mostly refuted (See e.g.\ \cites{gaanwekatten, keisler4, kano2}).  } by Bishop and Connes that 
\emph{Nonstandard Analysis be devoid of computational/constructive content}, we believe that the word `surprise' is perhaps not misplaced to describe our results.  

\medskip

In a nutshell, to obtain the previous theorem, one first establishes the `nonstandard uniform' version \eqref{wagner} of $\MCT_{\ns}\di \paai$, which yields the `super-pointwise' version \eqref{kkkk}.  The latter is then weakened into \eqref{froodcor33} and then weakened into \eqref{frood33}; this modification should be almost identical for other similar implications.  
In particular, it should be straightforward, but unfortunately beyond the page limit, to obtain versions of Theorems \ref{tochwelcrux} and \ref{sef33} for \emph{K\"onig's lemma} and \emph{Ramsey's theorem} (\cite{simpson2}*{III.7}), or any theorem equivalent to $\ACA_{0}$ in RM for that matter.  

\medskip

Furthemore, results related to \emph{weak K\"onig's lemma}, the third Big Five system of RM (\cite{simpson2}*{IV}) and the RM zoo (\cite{damirzoo}), can be obtained in the same way as above.  
For instance, one can easily obtain $\paai\di \WKL_{\ns}$ where the latter is the nonstandard modification of $\WKL$ stating the existence of a \emph{standard} path for every \emph{standard} infinite binary tree.  However, the existence of a `textbook proof' (as discussed right below Theorem \ref{tochwelcrux}) for this implication (or the reverse implication) leads to a contradiction.        

\medskip
  
In conclusion, higher-order computability results can be obtained from arbitrary proofs of $\MCT_{\ns}\di \paai$, while the \emph{textbook proof} as in the proof of Theorem \ref{partje} yields classical computability theory as in Theorem \ref{sef33}.

\subsection{The connection between higher-order and classical computability theory}
This paper would not be complete without a discussion of the $\ECF$-translation, which connects higher-order and second-order mathematics.
In particular, we show that applying the $\ECF$-translation to e.g.\ \eqref{frood} does not yield e.g.\ \eqref{froodcor33}.  

\medskip
  
We first define the central $\ECF$-notion of `associate' which some will know in an equivalent guise:  
Kohlenbach shows in \cite{kohlenbach4}*{Prop.\ 4.4} that the existence of a `RM code' for a continuous functional $\Phi^{2}$ as in \cite{simpson2}*{II.6.1}, is equivalent to the 
existence of an \emph{associate} for $\Phi$, and equivalent to the existence of a modulus of continuity for $\Phi$, Simpson's claims from \cite{simpson2}*{I.8.9.5} notwithstanding.    
\bdefi\label{defke}
The function $\alpha^{1}$ is an \emph{associate} of the continuous $\Phi^{2}$ if:
\begin{enumerate}[(i)]
\item $(\forall \beta^{1})(\exists k^{0})\alpha(\overline{\beta} k)>0$,\label{defitem2}
\item $(\forall \beta^{1}, k^{0})(\alpha(\overline{\beta} k)>0 \di \Phi(\beta)+1=_{0}\alpha(\overline{\beta} k))$.\label{defitem}
\end{enumerate}
\edefi
\noindent
With regard to notation, it is common to write $\alpha(\beta)$, to be understood as $\alpha(\overline{\beta} k)-1$ for large enough $k^{0}$ (See also Definition \ref{depke} below).  
Furthermore, we assume that every associate is a \emph{neighbourhood function} as in \cite{kohlenbach4}, i.e.\ $\alpha$ also satisfies 
\[
(\forall \sigma^{0^{*}}, \tau^{0^{*}})\big[{\alpha}(\sigma)>0\wedge |\sigma|\leq |\tau| \wedge (\forall i<|\sigma|)(\sigma(i)=\tau(i)) \di \alpha(\sigma)=\alpha(\tau) \big].  
\]  
We now sketch the $\ECF$-translation;
Note that $\RCAo$ is Kohlenbach's base theory for higher-order RM (\cite{kohlenbach2}); this system is essentially $\textsf{E-PA}^{\omega}$ weakened to one-quantifier-induction and with a fragment of the axiom of choice.  
\begin{rem}[$\ECF$-translation]\label{klit}\rm
The syntactic translation `$[\cdot]_{\ECF}$' is introduced in \cite{troelstra1}*{\S2.6.5} and we refer to the latter for the exact definition.  
Intuitively, applying the $\ECF$-translation to a formula amounts to nothing more than \emph{replacing all objects of type two or higher by associates}.  
Furthermore, Kohlenbach observes in \cite{kohlenbach2}*{\S2} that if $\RCAo\vdash A$ then $\RCA_{0}^{2}\vdash [A]_{\ECF}$, i.e.\ $[\cdot]_{\ECF}$ provides a translation from $\RCAo$ to (a system which is essentially) $\RCA_{0}$, the base theory of RM.  
\end{rem}
Thus, we observe that the $\ECF$-translation connects higher-order and second-order mathematics.
We now show that the $\ECF$-translation is not a `magic bullet' in that $[A]_{\ECF}$ may not always be very meaningful, as discussed next.     
\begin{exa}[The $\ECF$-translation of $(\mu^{2})$]\label{klit2}\rm 
The $\ECF$-translation interprets the \emph{discontinuous}\footnote{Suppose $f_{1}=11\dots$ and $\mu^{2}$ from $(\mu^{2})$ is continuous; then there is $N_{0}^{0}$ such that $(\forall g^{1})(\overline{f_{1}}N_{0}=\overline{g}N_{0}\di \mu(f_{1})=_{0}\mu(g))$.  Let $N_{1}$ be the maximum of $N_{0}$ and $\mu(f_{1})$.  Then $g_{0}:=\overline{f_{1}}N_{1}*00\dots$ satisfies $\overline{f_{1}}N_{1}=\overline{g_{0}}N_{1}$, and hence $\mu(f_{1})=\mu(g_{0})$ and $f_{1}(\mu(f_{1}))=g_{0}(\mu(g_{0}))$, but the latter is $0$ by the definition of $g_{0}$ and $\mu$, a contradiction.\label{korko}} functional $\mu^{2}$ as in $\MU(\mu)$ as a \emph{continuous} object satisfying the latter formula, which is of course impossible\footnote{If a functional has an associate, it must be continuous on Baire space.  We established in Footnote \ref{korko} that $(\mu^{2})$ cannot be continuous, and thus cannot have an associate.}, and the same holds for theorems equivalent to $(\mu^{2})$ as they involve discontinuous functionals as well.  
Hence, the $\ECF$-translation reduces the implications in \eqref{frood} to (correct) trivialities of the form `$0=1\di 0=1$'. 
\end{exa}
By the previous example, we observe that the answer to question \textsf{(Q)} is not just `apply $\ECF$' in the case of theorems involving $(\mu^{2})$.
Nonetheless, we \emph{could} apply the $\ECF$-translation to \eqref{frood33} and \eqref{froodcor33} to replace the terms $s, u, v$ by \emph{associates}.  
To this end, we require definition of partial function application (See e.g.\ \cite{troelstra1}*{1.9.12} or \cite{kohlenbach3}*{Def.\ 3.58}) for the final corollary.    
\bdefi[Partial function application]\label{depke}
For $\alpha^{1}, \beta^{1}$, `$\alpha(\beta)$' is defined as
\[
\alpha(\beta):=
\begin{cases}
\alpha(\overline{\beta}k)-1 & \textup{If $k^{0}$ is the least $n$ with $\alpha(\overline{\beta}n)>0$}\\
\textup{undefined} & \textup{otherwise}
\end{cases},
\]
and $\alpha | \beta := (\lambda n^{0})\alpha(\langle n\rangle*\beta)$.  
We write $\alpha(\beta)\downarrow$ to denote that $\alpha(\beta)$ is defined, and $\alpha|\beta\downarrow$ to denote that $(\alpha|\beta)(n)$ is defined for all $n^{0}$.  For $\beta^{1}, \gamma^{1}$, we define the paired sequence $\beta\oplus \gamma$ by putting $(\beta\oplus \gamma)(2k)=\beta(k)$ and $(\beta\oplus \gamma)(2k+1)=\gamma(k)$.   
\end{defi}
We now consider the following corollary to Theorem \ref{sef33}.  
\begin{cor}\label{sef3xxxx3}
From the textbook proof of $\MCT_{\ns}\di \paai$, a term $z^{1}$ can be extracted such that $\textsf{\textup{E-PA}}^{\omega*}$ proves:
\be\label{froodxxxxxx33}
(\forall \psi^{1}, A^{1})\big[ \MCT_{\ef}^{A}(\psi)\di  [z|(\psi\oplus A)\downarrow \wedge~ \MU^{A}(z|(\psi\oplus A))].
\ee
\end{cor}
\begin{proof}
Immediate from applying the $\ECF$-translation to \eqref{frood33}.  
\end{proof}
Note that \eqref{froodxxxxxx33} is part of second-order arithmetic.  
\begin{ack}\rm
The author would like to thank Richard Shore, Anil Nerode, and Vasco Brattka for their valuable advice and encouragement.
\end{ack}

\bye